\documentclass[12pt]{article}

\usepackage{amsthm,amsmath,amssymb,amsfonts}
\RequirePackage[numbers]{natbib}
\RequirePackage[colorlinks,citecolor=blue,urlcolor=blue]{hyperref}

\RequirePackage{hypernat}

\renewcommand{\Pr}{\mathbf{P}}
\newtheorem{lemm}{Lemma}
\newtheorem{theo}{Theorem}
\newtheorem{coro}{Corollary}
\newtheorem{prop}{Proposition}
\newcommand{\trp}{\mbox{$^{\tt T}$}}

\begin{document}

\title{Asymptotics for a Class of Dynamic Recurrent Event Models}

\author{Edsel A.\ Pe{\~n}a\footnote{Research partially supported by NSF Grants DMS0805809 and DMS1106435 and
NIH Grants  R01CA154731, P20RR17698, and P30GM103336-01A1.}
\medskip \\ Department of Statistics \\ University of South Carolina \\
Columbia, SC 29208 USA}

\maketitle

\begin{abstract}
Asymptotic properties, both consistency and weak convergence,
of estimators arising in a general class of
dynamic recurrent event models are presented. The class of models take
into account the impact of interventions after each event occurrence,
the impact of accumulating event occurrences, the induced informative
and dependent right-censoring mechanism due to the data-accrual scheme,
and the effect of covariate processes on the recurrent event occurrences.
The class of models subsumes as special cases many of the recurrent event
models that have been considered in biostatistics, reliability, and in the
social sciences. The asymptotic properties presented have the potential of
being useful in developing goodness-of-fit and model validation procedures,
confidence intervals and confidence bands constructions, and hypothesis
testing procedures for the finite- and infinite-dimensional
parameters of a general class of dynamic recurrent event models, albeit the models
without frailties.

\medskip


\medskip

\end{abstract}

\section{Introduction and Background}
\label{sec-intro}

Recurrent events pervade many disciplines such as the biomedical and public health
sciences, engineering sciences, social and political sciences, economic sciences,
and even sporting events. Examples of such events are non-fatal heart attacks,
hospitalization of a patient with a chronic disease, migraines, breakdown of an
electronic or mechanical system, discovery of a bug in a software, disagreement in
a marriage, change of a job, Dow Jones Industrial Average (DJIA) decreasing by at least
200 points during a trading day, a perfect baseball game in the Major Leagues,
a goal scored in a World Cup soccer game, and
many others. The mathematical modeling of recurrent events, together with the development
of statistical inference procedures for the models, are of paramount
importance.

There are two approaches to the specification of mathematical models for recurrent
events. The first is a {\em full} specification of the probability measure on the measurable
space induced by the monitoring of the recurrent event. This is done by specifying the
joint distributions of the calendar times of event occurrences, or equivalently the
joint distributions of the inter-event times. Alternatively, the probability measure
can be specified as a measure on the space of paths of the stochastic process arising
from the monitoring of the recurrent event. The simplest and perhaps most common full
parametric model is when the counting process associated with the event accrual is
assumed to follow a homogeneous Poisson process (HPP), in which case the inter-event
times are independent and identically distributed (IID) with common negative exponential
distribution. One may also specify a nonparametric model by simply assuming that the
inter-event time distribution is some unknown continuous distribution, resulting in the
IID renewal model. The general dynamic model of interest in this article is of the
full model variety.

The second modeling approach is referred to as {\em marginal} modeling.
In its basic form, the event position within a unit is utilized as a stratifying variable,
and a (marginal) probability measure is specified for each of the resulting strata.
This approach was pioneered in the papers \cite{PreWilPet81, WeiLinWei89}. It
should be observed that the class of full models subsumes the class of marginal models.
However, proponents of the marginal modeling approach espouse this marginal approach since it
generally leads to an easier interpretation of model parameters though, at the same time,
it may be difficult to justify a full model which is consistent with the specified
marginal models. In fact, there could be several full models that are consistent with
the marginal models.

An IID distributional specification for the inter-event times is clearly
an oversimplification since it will often be the case that after an event
occurrence some type of intervention, such as a corrective measure or a
repair, will be performed, thereby altering the distribution of the time to
the next event occurrence. Furthermore, time-dependent concomitant variables
could also impact the distributions of the inter-event times, and within
a unit the inter-event times may be correlated owing to unobserved
latent variables. The number of event occurrences could also impact these
distributions, such as when event occurrences weakens the unit, thereby
stochastically shortening the time to the next event occurrence.
Due to practical and unavoidable constraints, the monitoring of the event
could also only be performed over a finite, possibly random, observation
window, and thus a {\em sum-quota accrual scheme} ensues wherein the number of
observed event occurrences is a random variable which is informative about
the event occurrence mechanism. This finite monitoring constraint also
produces a right-censored observation, which could not be ignored in performing
inference because of selection bias issues. The class of dynamic recurrent
event models proposed in \cite{PenHol04} incorporates the above considerations.
This class of models is a specific member of the class of models of interest
in this article. The major goals of this article are to obtain the asymptotic
properties of semi-parametric estimators of the model parameters for the general
class of dynamic recurrent event time models of the type in \cite{PenHol04}.
Note that algorithmic issues of the semi-parametric estimators for the model in
\cite{PenHol04} were dealt with in \cite{PenSlaGon07}.

This article focuses on the large-sample properties of semiparametric estimators
for the parameters of the class of dynamic models described in section \ref{sec-models}.
These semiparametric estimators are described in section \ref{sec-estimators}.
Consistency properties of the estimators will be established in section
\ref{sec-consistency}, while weak convergence properties will be developed in
section \ref{sec-weak convergence}.

\section{Class of Dynamic Models}
\label{sec-models}

In this section we describe the general class of dynamic models of interest.
In the sequel, $(\Omega,\mathcal{F},\Pr)$ is the basic probability space on which
all random entities are defined.
Consider a unit that is monitored over the calendar time $[0,s^*]$,
where $s^* \in (0,\infty)$ is a fixed calendar time.
We suppose that for this unit there is a $1 \times q$ vector of possibly
time-varying bounded covariates
%
$\mathbf{X} = \left\{
\mathbf{X}(s):\ s \in [0,s^*]\right\}.$
%
We shall denote by
%
$\mathbf{N}^\dagger = \left\{
N^\dagger(s):\ s \in [0,s^*]\right\}$
%
the counting process such that $N^\dagger(s)$ is the number of event occurrences over
the period $[0,s]$. The at-risk process will be
%
$\mathbf{Y}^\dagger = \left\{
Y^\dagger(s):\ s \in [0,s^*]\right\},$
%
so that $Y^\dagger(s)$ indicates whether the unit is still under observation, i.e.,
at-risk, at time $s$. This will usually be defined via
%
$Y^\dagger(s) = I\{\tau \ge s\},$
%
where $I\{\cdot\}$ is the indicator function and $\tau$ is some positive-valued
random variable. We shall denote by
%
$\mathfrak{F} = \left\{
\mathcal{F}_s:\ s \in [0,s^*]
\right\}$
%
a filtration on $(\Omega,\mathcal{F},\Pr)$ such that $\mathbf{N}^\dagger$,
$\mathbf{Y}^\dagger$, and $\mathbf{X}$ are $\mathfrak{F}$-adapted  and, in addition,
$\mathbf{Y}^\dagger$ and $\mathbf{X}$ are also $\mathfrak{F}$-predictable.

The class of dynamic models of interest postulates that for $k \in \{1,2,\ldots\}$
and with $dN^\dagger(s) \equiv N^\dagger((s+ds)-) - N^\dagger(s-)$, as $ds \downarrow 0$
and for $s \in [0,s^*)$,
\begin{eqnarray}
\lefteqn{ \Pr\{dN^\dagger(s) \ge k | \mathcal{F}_{s-}\} \nonumber} \\ & = &
\left[Y^\dagger(s) \lambda(s|\mathbf{X}(s))
I\{k=1\} + o_p(1) I\{k \ge 2\}\right] ds,\ \mbox{a.e.-$[\Pr]$},
\label{intensity for a unit}
\end{eqnarray}
where
\begin{equation}
\label{form of intensity for a unit}
\lambda(s|\mathbf{X}(s)) =
\lambda_0[\mathcal{E}(s)] \rho[s,N^\dagger(s-);\alpha]
\psi[\mathbf{X}(s) \mathbf{\beta}]
\end{equation}
and with $\mathfrak{E} = \{\mathcal{E}(s):\ s \in [0,s^*]\}$
being an $\mathfrak{F}$-predictable process with paths that are
piecewise left-continuous, nonnegative, $\mathcal{E}(s) \le s$,
and piecewise differentiable with derivative satisfying
$\mathcal{E}^\prime(s) \ge 0$; $\lambda_0(\cdot)$ is an unknown
baseline hazard rate function with cumulative hazard function
$\Lambda_0(\cdot) = \int_0^\cdot \lambda_0(s) ds
\in \mathcal{C}$; $\rho(\cdot,\cdot;\alpha)$ is a known nonnegative bounded
function over $\Re_+ \times \mathbf{Z}_+ = \{0,1,2,\ldots\}$ with
$\rho(s,0;\alpha) = 1$ and $\alpha \in \Re^q$ is an unknown $q$-dimensional
parameter; and $\psi(\cdot)$ is a known nonnegative link function on $\Re$ and
with $\beta \in \Re^p$ an unknown $p$-dimensional regression parameter.
The process $\mathfrak{E}$ is called the {\em effective} age process.
We shall assume that $\tau \sim G(\cdot)$, where $G(\cdot)$ is some distribution
function which does not involve $(\lambda_0(\cdot),\alpha,\beta)$,
hence it is considered a nuisance parameter.
The regressors $\mathbf{X}$ is a vector-valued bounded and
predictable process whose probabilistic
structure may also contain some unknown nuisance parameters. A technical condition
that we will assume (see the paper \cite{GjeRoyPenAal10}) is that the counting
process $\mathbf{N}^\dagger$ is non-explosive over $[0,s^*]$, that is,
%
$\Pr\left\{N^\dagger(s^*) < \infty\right\} = 1.$
%
This condition necessarily imposes a constraint on the form of the function
$\rho(\cdot,\cdot;\cdot)$ and the model parameters.

The model parameter of main interest is
\begin{equation}
\label{model parameter}
\mathbf{\theta} = (\Lambda_0(\cdot),\alpha,\beta) \in \Theta \equiv \mathcal{C} \times
\Re^q \times \Re^p
\end{equation}
where $\mathcal{C}$ is some class of cumulative hazard functions on $\Re_+$,
which will typically be a nonparametric class. Thus, $\theta$ will be
a semiparametric parameter.
Defining the process $\mathbf{M}^\dagger = \{M^\dagger(s;\theta):\
s \in [0,s^*]\}$ with
%
$M^\dagger(s;\theta) = N^\dagger(s) - A^\dagger(s;\theta)$
%
and where
\begin{equation}
\label{compensator for unit}
A^\dagger(s;\theta) = \int_0^s Y^\dagger(v)
\lambda_0[\mathcal{E}(v)] \rho[v,N^\dagger(v-);\alpha]
\psi[\mathbf{X}(v) \mathbf{\beta}] dv,
\end{equation}
the model is tantamount to the condition that $\mathbf{M}^\dagger$ is
a zero-mean square-integrable $\mathfrak{F}$-martingale.
The model specified in (\ref{intensity for a unit}) and (\ref{form of intensity for a unit})
is a slightly more general version of those in \cite{PenHol04} and \cite{PenSlaGon07} since
we allow the $\rho$-function to also directly depend on $s$ aside from $N^\dagger(s-)$.
For more background about this class of models and many specific models subsumed by this class of
models, see \cite{PenHol04,PenSlaGon07}. This general
class of models includes as special cases models that have been considered in the
biostatistics and reliability settings. To mention two specific models, if
$\mathcal{E}(s) = s - S_{N^\dagger(s-)}$ with $0 = S_0 < S_1 < S_2 < \ldots$ being
the times of successive event occurrences, so $\mathcal{E}(\cdot)$ represents
the backward recurrence time function, the model coincides with resetting the
age of the unit to zero after each event occurrence,
which is referred to in the reliability literature as a perfect repair;
while if we have $\mathcal{E}(s) = s$, then we say that a minimal
repair is performed after each event occurrence. If the latter specification
is further coupled with $\rho(v,k;\alpha) = 1$, then we recover the
Andersen-Gill multiplicative intensity model \cite{AndGil82}; also the
Cox proportional hazards (PH) model \cite{Cox72} when $\psi(v) = \exp(v)$.

We consider the situation where $n$ IID copies
%
$\mathfrak{D}_n \equiv \mathfrak{D} = (\mathbf{D}_1, \mathbf{D}_2, \ldots, \mathbf{D}_n)$
%
of the basic observable
%
$\mathbf{D} =
(\mathbf{N}^\dagger, \mathbf{Y}^\dagger, \mathcal{E}, \mathbf{X})$
%
are observed. We denote by $\mathcal{D}$ the sample space of $\mathbf{D}$, so
that the sample space for $\mathfrak{D}$ is $\mathcal{D}^n$.
A larger filtration on $(\Omega, \mathcal{F}, \Pr)$ is formed from
the $n$ unit filtrations according to
\begin{displaymath}
\mathfrak{F} = \bigvee_{i=1}^n \mathfrak{F}_i =
\sigma\left(\bigcup_{i=1}^n \mathfrak{F}_i\right).
\end{displaymath}
Inference on the model parameter $\theta = (\Lambda_0(\cdot),\alpha,\beta)$, or
relevant functionals of $\theta$, are to be based on the realization of $\mathfrak{D}_n$.
Properties of the inferential procedures are to be examined when $n \rightarrow \infty$.

We shall use functional notation in the sequel.
Thus, for a possibly vector-valued function $g$ defined on $\mathcal{D}$,
$\mathbf{P}g$ will represent the theoretical expectation of $g(\mathbf{D})$,
while $\mathbb{P} g \equiv \mathbb{P}_n g$ will represent the empirical expectation of
$g$ given $\mathfrak{D}_n$. That is,
$\mathbf{P} g = \int g(\mathbf{d}) \mathbf{P}(d\mathbf{d})$
and
$\mathbb{P} g = \frac{1}{n} \sum_{i=1}^n g(\mathbf{D}_i).$
The theoretical and empirical covariances of $g$ are defined, respectively, via
$\mathbf{V}g = \mathbf{P}(g - \mathbf{P}g)^{\otimes 2}$
and
$\mathbb{V}g = \mathbb{P} (g - \mathbb{P}g)^{\otimes 2},$
where, for a column vector $a$, we write $a^{\otimes 0} = 1$, $a^{\otimes 1} = a$, and
$a^{\otimes 2} = aa\trp$.

\section{Semiparametric Estimators}
\label{sec-estimators}

\subsection{Doubly-Indexed Processes}

The intensity model in (\ref{compensator for unit}) has the distinctive feature that
the baseline hazard rate $\lambda_0(\cdot)$ is evaluated at time $s$ at the effective
age $\mathcal{E}(s)$. Since of interest is to infer about $\lambda_0(\cdot)$ or
$\Lambda_0(\cdot)$, we need to de-couple $\lambda_0(\cdot)$ from $\mathcal{E}(\cdot)$.
As demonstrated in \cite{Sel88,PenStrHol01,PenSlaGon07} such de-coupling is facilitated
through the use of doubly-indexed processes.

Let $t^* \in (0,\infty)$ be fixed, and define $\mathcal{S} = [0,s^*]$
and $\mathcal{T} = [0,t^*]$. Form $\mathcal{I} = \mathcal{S} \times \mathcal{T}$.
For our purpose we define the following $\mathcal{I}$-indexed processes
associated with the $(N^\dagger, Y^\dagger, \mathcal{E}, \mathbf{X})$ processes
for one unit:
$\mathbf{Z} = \{Z(s,t): (s,t) \in \mathcal{I}\}$,
$\mathbf{N} = \{N(s,t): (s,t) \in \mathcal{I}\}$,
$\mathbf{A} = \{A(s,t;\theta): (s,t) \in \mathcal{I}\}$, and
$\mathbf{M} = \{M(s,t;\theta): (s,t) \in \mathcal{I}\}$, where
\begin{eqnarray*}
& Z(s,t) = I\{\mathcal{E}(s) \le t\}; & \\
& N(s,t) = \int_0^s Z(v,t) N^\dagger(dv); & \\
& A(s,t;\theta) = \int_0^s Z(v,t) A^\dagger(dv;\theta); & \\
& M(s,t;\theta) = N(s,t) - A(s,t;\theta) = \int_0^s Z(v,t) M^\dagger(dv;\theta). &
\end{eqnarray*}
As an interpretation, note that $N(s,t)$ is the number of occurrences of the recurrent
event over the period $[0,s]$ and for which the effective ages on these occurrences are
at most $t$. We introduce the following notation: For a finite subset $\mathbf{T} \subset
\mathcal{T}$, $N(\cdot,\mathbf{T}) \equiv (N(\cdot,t): t \in \mathbf{T})$ and similarly
for the other processes.

\begin{prop}
\label{prop-mart property}
Let $\mathbf{T} \subset \mathcal{T}$ be a finite set. Then $\{M(s,\mathbf{T};\theta):
s \in \mathcal{S}\}$ is a $|\mathbf{T}|$-dimensional zero-mean square-integrable martingale
with predictable quadratic covariation process
$$\langle M(\cdot,\mathbf{T};\theta) \rangle(s) =
\left[
\left(
A(s,\min(t_1,t_2);\theta)
\right)_{t_1,t_2\in\mathbf{T}}
\right], \ s \in \mathcal{S}.$$
Consequently,
%
$\mathbf{P} N(s,\mathbf{T}) = \mathbf{P} A(s,\mathbf{T};\theta)$
and
$\mathbf{V} M(s,\mathbf{T};\theta) =
\mathbf{P} \langle M(\cdot,\mathbf{T};\theta) \rangle(s).$
%
\end{prop}

\begin{proof}
Follows from the boundedness and predictability of $s \mapsto Z(s,\mathbf{T})$, the fact that
$Z(s,t_1) Z(s,t_2) = Z(s,\min(t_1,t_2))$, by stochastic
integration theory, and since
$M(s,\mathbf{T};\theta) = \int_0^s Z(v,\mathbf{T}) M^\dagger(dv;\theta)$.
\end{proof}

Let $s \in \mathcal{S}$ and denote by
\begin{displaymath}
0 \equiv S_0 < S_1 < S_2 < \ldots < S_{N^\dagger(s-)} < S_{N^\dagger(s-)+1} \equiv
\min(s,\tau)
\end{displaymath}
the $N^\dagger(s-)$ successive event occurrence times for the unit.
Define the (random) functions
$\mathcal{E}_{j}: \mathcal{S} \rightarrow \Re$ via
\begin{displaymath}
\mathcal{E}_{j}(v) = \mathcal{E}(v) I_{(S_{j-1},S_j]}(v)
\end{displaymath}
for $j = 1, 2, \ldots, N^\dagger(s-)+1.$
By condition, on $(S_{j-1},S_j)$, $\mathcal{E}_{j}(\cdot)$ is nondecreasing
and differentiable. We denote by $\mathcal{E}_{j}^{-1}(\cdot)$ its inverse
function and by $\mathcal{E}_{j}^\prime(\cdot)$ its derivative.
Define the (random) functions
$\varphi_{j}: \mathcal{S} \rightarrow \Re$ according to
\begin{displaymath}
\varphi_{j}(v;\alpha,\beta) =
\frac{\rho(v,j-1;\alpha) \psi[\mathbf{X}(v)\beta]}
{\mathcal{E}_{j}^\prime(v)}
I_{(S_{j-1},S_j]}(v),
\end{displaymath}
for $j = 1, 2, \ldots, N^\dagger(s-)+1$. Next, we define the doubly-indexed process
$\mathbf{Y} = \{Y(s,t;\alpha,\beta): (s,t) \in \mathcal{I}\}$ according to
\begin{equation}
\label{generalized at-risk process}
Y(s,t;\alpha,\beta) =
\sum_{j=1}^{N^\dagger(s-)+1}
\varphi_{j}[\mathcal{E}_{j}^{-1}(t);\alpha,\beta]
I_{(\mathcal{E}_{j}(S_{j-1}), \mathcal{E}_{j}(S_j)]}(t).
\end{equation}
This is a generalized at-risk process. The importance of these doubly-indexed
processes arise from the representation of the $\mathbf{A}$-process in
Proposition \ref{prop-rep of A process},
which de-couples the effective age process $\mathcal{E}(\cdot)$ from the baseline
hazard function $\Lambda_0(\cdot)$, and the change-of-variable identity in
Proposition \ref{prop-change of variable}. Restricted forms of these results
were used in the IID recurrent event model considered in \cite{PenStrHol00,PenStrHol01}.

\begin{prop}
\label{prop-rep of A process}
For $(s,t) \in \mathcal{I}$, $A(s,t;\theta) = \int_0^t Y(s,w;\alpha,\beta) \Lambda_0(dw).$
\end{prop}

\begin{proof}
Partition the region of integration $(0,s]$ into the disjoint union
$(0, s] = \cup_{j=1}^{N^\dagger(s-)+1} (S_{j-1},S_j];$
do a variable transformation on each region; manipulate; and then simplify.
\end{proof}

\begin{prop}
\label{prop-change of variable}
Let $\{H(s,t): (s,t) \in \mathcal{I}\}$ be a bounded vector-valued process such that for each
$t$, $s \mapsto H(s,t)$ is predictable. For $(s,t) \in \mathcal{I}$, we have
\begin{displaymath}
\int_0^s H(s,\mathcal{E}(v)) M(dv,t) = \int_0^t H(s,w) M(s,dw).
\end{displaymath}
\end{prop}

\begin{proof}
Start with the left-hand side, write the $M$ process into its $N^\dagger$ and
$A^\dagger$ components, then perform the same manipulations as in the proof of
Proposition \ref{prop-rep of A process}.
\end{proof}

\subsection{Estimation of $\Lambda_0$}

Propositions \ref{prop-mart property} and \ref{prop-rep of A process} now combine
to suggest the stochastic differential equation, for an observable $\mathbf{D}$,
\begin{displaymath}
N(s^*,dt) = Y(s^*,t;\alpha,\beta) \Lambda_0(dt) + M(s^*,dt;\theta).
\end{displaymath}
When data $\mathfrak{D}_n$ is available from $n$ units, we therefore obtain the
differential form
\begin{equation}
\label{diff form}
\mathbb{P} N(s^*,dt) =
\{\mathbb{P} Y(s^*,t;\alpha,\beta)\} \Lambda_0(dt) +
\mathbb{P} M(s^*,dt;\theta).
\end{equation}
Define
\begin{equation}
\label{S0}
S^{(0)}(s,t;\alpha,\beta) = \mathbb{P} Y(s,t;\alpha,\beta)
\equiv \frac{1}{n}\sum_{i=1}^n Y_i(s,t;\alpha,\beta)
\end{equation}
and
%
$J(s,t;\alpha,\beta) = I\{S^{(0)}(s,t;\alpha,\beta) > 0\}.$
%
With the convention that $0/0=0$, we obtain from (\ref{diff form}) the
stochastic integral identity
\begin{eqnarray}
\lefteqn{
\int_0^t \frac{J(s^*,w;\alpha,\beta)}{S^{(0)}(s^*,w;\alpha,\beta)} \mathbb{P} N(s^*,dw) }
\nonumber \\ &  = &
\int_0^t J(s^*,w;\alpha,\beta) \Lambda_0(dw) +
\int_0^t \frac{J(s^*,w;\alpha,\beta)}{S^{(0)}(s^*,w;\alpha,\beta)}
\mathbb{P} M(s^*,dw;\theta).
\label{inte form}
\end{eqnarray}
Let us consider the last term in (\ref{inte form}). We have
\begin{eqnarray}
\lefteqn{ \int_0^t \frac{J(s^*,w;\alpha,\beta)}{S^{(0)}(s^*,w;\alpha,\beta)}
\mathbb{P} M(s^*,dw;\theta) } \nonumber \\
& = & \frac{1}{n} \sum_{i=1}^n \int_0^t
\frac{J(s^*,w;\alpha,\beta)}{S^{(0)}(s^*,w;\alpha,\beta)} M_i(s^*,dw;\theta) \nonumber \\
& = & \frac{1}{n} \sum_{i=1}^n \int_0^{s^*}
\frac{J(s^*,\mathcal{E}_i(v);\alpha,\beta)}
{S^{(0)}(s^*,\mathcal{E}_i(v);\alpha,\beta)} M_i(dv,t;\theta)
\label{sum of ints}
\end{eqnarray}
where the last equality is obtained by invoking Proposition \ref{prop-change of variable}.
The integrand in each summand in (\ref{sum of ints}) is bounded and predictable, so it
follows from stochastic integration theory that, for $i = 1, 2, \ldots, n$,
\begin{equation}
\label{mean zero property of ints}
\mathbf{P} \int_0^{s^*}
\frac{J(s^*,\mathcal{E}_i(v);\alpha,\beta)}
{S^{(0)}(s^*,\mathcal{E}_i(v);\alpha,\beta)} M_i(dv,t;\theta) = 0.
\end{equation}
It therefore follows from (\ref{inte form}) and (\ref{mean zero property of ints}) that
\begin{displaymath}
\mathbf{P} \int_0^t \frac{J(s^*,w;\alpha,\beta)}{S^{(0)}(s^*,w;\alpha,\beta)}
\mathbb{P} N(s^*,dw) =
\mathbf{P} \int_0^t J(s^*,w;\alpha,\beta) \Lambda_0(dw).
\end{displaymath}
Analogously to Aalen's idea \cite{Aal78}, if for the moment we assume that $(\alpha,\beta)$
is known, we may propose a method-of-moments estimator for $\Lambda_0(\cdot)$ given by
\begin{equation}
\label{GNAE1}
\tilde{\Lambda}_0(t;\alpha,\beta) =
\int_0^t \frac{J(s^*,w;\alpha,\beta)}{S^{(0)}(s^*,w;\alpha,\beta)}
\mathbb{P} N(s^*,dw) =
\int_0^t \frac{\mathbb{P} N(s^*,dw)}{S^{(0)}(s^*,w;\alpha,\beta)}.
\end{equation}
However, $(\alpha,\beta)$ is not known, hence $\tilde{\Lambda}_0$ is not an estimator.
We now therefore find an estimator for $(\alpha,\beta)$, which will then be plugged-in
(\ref{GNAE1}) to obtain a legitimate estimator of $\Lambda_0$.

\subsection{Estimator of $(\alpha,\beta)$}

For the purpose of estimating $(\alpha,\beta)$, we form a generalized likelihood process,
based on $\mathfrak{D}_n$, denoted by $\mathbf{L} = \{L(s,t;\theta): (s,t) \in \mathcal{I}\}$.
We define
\begin{displaymath}
L(s,t;\theta) = \prod_{i=1}^n \prod_{v=0}^s
\left[
A_i(dv,t;\theta)
\right]^{N_i(\Delta v,t)}
\left[
1 - A_i(dv,t;\theta)
\right]^{1-N_i(\Delta v,t)},
\end{displaymath}
with the understanding that when the product operation is over a continuous index, such as $v$
in the second product operation, then it means product-integral; see \cite{GilJoh89}.
By property of the product-integral and re-writing in an expanded form, we have that
\begin{displaymath}
L(s,t;\theta) =
\left\{
\prod_{i=1}^n \prod_{v=0}^s
\left[
Z_i(v,t) A_i^\dagger(dv;\theta)
\right]^{Z_i(v,t) N_i^\dagger(\Delta v)}
\right\}
\exp\left\{-n \mathbb{P} A(s,t;\theta)\right\}.
\end{displaymath}
This likelihood process involves the functional parameter $\Lambda_0(\cdot)$, for
which we have an estimator given in (\ref{GNAE1}) if $(\alpha,\beta)$ is known.
We can therefore obtain a profile likelihood for $(\alpha,\beta)$ by replacing
the $\Lambda_0(\cdot)$ in $L(s^*,t^*;\theta)$ by the
$\tilde{\Lambda}_0(s^*,\cdot;\alpha,\beta)$ in (\ref{GNAE1}). Doing so yields a
profile likelihood function given by
\begin{equation}
\label{profile likelihood}
L_P(s^*,t^*;\alpha,\beta) =
\prod_{i=1}^n \prod_{v=0}^{s^*}
\left[
\frac{\rho(v,N_i^\dagger(v-);\alpha) \psi[\mathbf{X}_i(v)\beta]}
{S^{(0)}(s^*,\mathcal{E}_i(v);\alpha,\beta)}
\right]^{N_i(\Delta v,t^*)}.
\end{equation}
This function may also be viewed as a generalized partial likelihood function
for $(\alpha,\beta)$ being very much reminiscent of the Cox partial likelihood function;
see \cite{Cox72, Cox75, AndGil82, FleHar91, AndBorGilKei93, AalBorGje08}.
From this partial likelihood
function we obtain its maximizer as our estimator of $(\alpha,\beta)$, that is,
\begin{equation}
\label{estimator of alpha and beta}
(\hat{\alpha},\hat{\beta}) \equiv (\hat{\alpha}(s^*,t^*),\hat{\beta}(s^*,t^*)) =
\arg\max_{(\alpha,\beta) \in \Re^q \times \Re^p}
L_P(s^*,t^*;\alpha,\beta).
\end{equation}
Numerical methods, such as the Newton-Raphson algorithm, are needed to obtain the values
of $(\hat{\alpha},\hat{\beta})$, as has been done in \cite{PenSlaGon07}.

Having obtained an estimator of $(\alpha,\beta)$, now replace $(\alpha,\beta)$
in $\tilde{\Lambda}_0(s^*,t;\alpha,\beta)$ to obtain an estimator of $\Lambda_0(\cdot)$.
This resulting estimator of $\Lambda_0(\cdot)$ is given by
\begin{equation}
\label{ABN estimator}
\hat{\Lambda}_0(s^*,t) = \tilde{\Lambda}_0(s^*,t;\hat{\alpha},\hat{\beta}) =
\int_0^t \frac{\mathbb{P} N(s^*,dw)}{S^{(0)}(s^*,w;\hat{\alpha},\hat{\beta})}, \
t \in \mathcal{T}.
\end{equation}
Observe that the form of this estimator is analogous to the estimator of the baseline
hazard function in the Cox PH model \cite{Cox72, BreCro74, AndGil82},
hence it seems appropriate to refer to this as a generalized Aalen-Breslow-Nelson (ABN)
estimator.

Denoting by $F_0$ the distribution function associated with the baseline hazard function
$\Lambda_0$, then dictated by the product-integral representation of $F_0$ by $\Lambda_0$, we
are able to obtain a product-limit type estimator of the survivor function
$\bar{F}_0(t) = 1 - F_0(t)$ given by
\begin{equation}
\label{GPLE}
\hat{\bar{F}}_0(s^*,t) = \prod_{w=0}^t
\left[
1 - \frac{\mathbb{P} N(s^*,dw)}{S^{(0)}(s^*,w;\hat{\alpha},\hat{\beta})}
\right], \ t \in \mathcal{T}.
\end{equation}

Small to moderate sample size properties
of the estimators presented above were examined through
simulation studies in \cite{PenSlaGon07} for specific forms of the effective age process
$\mathcal{E}$, for a function $\rho$ which was made to depend on $s$
only through $N^\dagger(s-)$, and for an exponential link function $\psi$.
Applications of these estimators to some real data sets were
also presented in that paper. However, general asymptotic properties of these estimators are
still unavailable, and establishing the large-sample properties of these semiparametric
estimators is the {\em raison d'\^etre} of the current paper.

\section{Preliminaries for Asymptotics}
\label{sec-Preliminaries for Asymptotics}

For studying the large-sample properties of our semiparametric estimators,
it is first convenient to deal with the model where $A^\dagger$ in
(\ref{compensator for unit}) is of form
\begin{equation}
\label{simplified A dag}
A^\dagger(s;\eta) = \int_0^s Y^\dagger(v) \lambda_0[\mathcal{E}(v)] \kappa(v;\eta) dv.
\end{equation}
Here $\kappa = \{\kappa(s;\eta): s \in \mathcal{S}\}$ is a bounded and predictable process
and $\eta \in \Gamma$ with $\Gamma$ an open subset of $\Re^{k}$.
We assume that $\eta \mapsto \kappa(s;\eta)$ is twice-differentiable and we let
\begin{displaymath}
\stackrel{.}{\kappa}(s;\eta) = \nabla_\eta \kappa(s;\eta)
\quad \mbox{and} \quad
\stackrel{..}{\kappa}(s;\eta) = \nabla_{\eta\eta\trp} \kappa(s;\eta).
\end{displaymath}
Later to obtain the specific results for the model in
(\ref{compensator for unit}), we then simply identify $\eta$ with $(\alpha,\beta)$
and with
\begin{displaymath}
\kappa(s;\eta) = \rho(s,N^\dagger(s-);\alpha) \psi[\mathbf{X}(s)\beta].
\end{displaymath}
With the above simplification, for one unit monitored over $\mathcal{S} = [0,s^*]$,
we will then define
\begin{eqnarray*}
\varphi_j(v;\eta) &  = & \frac{\kappa(v;\eta)}{\mathcal{E}^\prime(v)} I_{(S_{j-1},S_j]}(v), \
j = 1,2,\ldots,N^\dagger(s^*-)+1; \\
Y(s^*,t;\eta) & = & \sum_{j=1}^{N^\dagger(s^*-)+1} \varphi_j[\mathcal{E}_j^{-1}(t);\eta]
I_{(\mathcal{E}(S_{j-1}), \mathcal{E}(S_j)]}(t),
\end{eqnarray*}
so that with $n$ units, we will then have
\begin{displaymath}
S^{(0)}(s^*,t;\eta)  =  \mathbb{P} Y(s^*,t;\eta) = \frac{1}{n}\sum_{i=1}^n Y_i(s^*,t;\eta)
\end{displaymath}
where in this last function the $\kappa$ functions may also depend on $i$.

We denote by $(\eta^0,\Lambda_0^0)$ the true parameter vector, and to simplify notation,
we suppress writing these true parameter vector in our functions if no confusion could arise.
Thus, $A_i(s^*,t) \equiv A_i(s^*,t;\eta^0,\Lambda_0^0)$, $Y_i(s^*,t) \equiv Y_i(s^*,t;\eta^0)$,
and $M_i(s^*,t) \equiv M_i(s^*,t;\eta^0,\Lambda_0^0)$.

In establishing consistency and weak convergence properties of the estimators, we will need
a general weak convergence result of processes formed as stochastic integrals of the processes
$M_i(s^*,t), i=1,2,\ldots,n$, which we recall are martingales with respect to $s^*$ but not
with respect to $t$.

Given an $n$ and an $(s^*, t, \eta)$, let us define a random discrete probability measure
$\mathbb{Q}_n(\cdot;s^*,t,\eta)$ on the (random) set $$\mathcal{K}_n(s^*) =
\{(i,j): j=1,2,\ldots,N_i^\dagger(s^*-)+1; i=1,2,\ldots,n\}$$ according to the probabilities
\begin{displaymath}
\mathbb{Q}_n({(i,j)};s^*,t,\eta) = \frac{1}{n}
\left\{
\frac{Y_i(s^*,t;\eta)}{S^{(0)}(s^*,t;\eta)}
\right\}
\left\{
\frac{\varphi_{ij}[\mathcal{E}_{ij}^{-1}(t);\eta]}{Y_i(s^*,t;\eta)}
I_{(\mathcal{E}_i(S_{ij-1}), \mathcal{E}_i(S_{ij})]}(t)
\right\}.
\end{displaymath}
For a function $g: \mathcal{K}_n(s^*) \rightarrow \Re^r$, which could be random and
also depending on $(s^*,t,\eta)$,
$$\mathbb{E}_{\mathbb{Q}_n(s^*,t,\eta)} g \equiv \mathbb{Q}_n(s^*,t,\eta)g$$
will denote its expectation with respect to the p.m.\ $\mathbb{Q}_n$
and
\begin{eqnarray*}
\mathbb{V}_{\mathbb{Q}_n(s^*,t,\eta)}g & \equiv &
\mathbb{Q}_n(s^*,t,\eta)[g - \mathbb{Q}_n(s^*,t,\eta)g]^{\otimes 2} \\ & = &
\mathbb{Q}_n(s^*,t,\eta)g^{\otimes 2} - [\mathbb{Q}_n(s^*,t,\eta)g]^{\otimes 2}
\end{eqnarray*}
will denote its variance-covariance matrix with respect to
$\mathbb{Q}_n$.

Let us also define
\begin{displaymath}
\mathbb{Q}_n(i;s^*,t,\eta) = \frac{1}{n} \left\{
\frac{Y_i(s^*,t;\eta)}{S^{(0)}(s^*,t;\eta)} \right\} =
\frac{Y_i(s^*,t;\eta)}{\sum_{l=1}^n Y_l(s^*,t;\eta)},
i=1,2,\ldots,n.
\end{displaymath}
Thus, when the function $g:  \mathcal{K}_n(s^*) \rightarrow \Re^r$
is such that $g(i,j) = g^*(i)$ for some $g^*$, then
\begin{eqnarray*}
\mathbb{Q}_n(s^*,t,\eta) g & = & \sum_{i=1}^n g^*(i)
\mathbb{Q}_n(i;s^*,t,\eta) = \sum_{i=1}^n g^*(i)
\left[\frac{Y_i(s^*,t;\eta)}{\sum_{l=1}^n Y_l(s^*,t;\eta)}\right].
\end{eqnarray*}
In this case, the variance-covariance matrix of $g$ with respect to
$\mathbb{Q}_n$ is also in more simplified form.

\begin{theo}
\label{theo-master theorem}
Let $\{H^{(n)}_i(s,t): (s,t) \in \mathcal{I}=[0,s^*] \times [0,t^*]\}$ for
$i=1,2,\ldots,n;$ $n=1,2,\ldots$ be a triangular array of vector processes,
and assume the following conditions:
\begin{itemize}
\item[(a)]
$\forall i$, $H_i^{(n)}$ is bounded and $\forall v \in [0,s]$, $H_i^{(n)}(s,\mathcal{E}_i(v))$
is $\mathbb{F}$-predictable;
\item[(b)]
There exists a deterministic function $s^{(0)}: \mathcal{I} \rightarrow \Re_+$ such that
$$|S^{(0)}(s^*,t) -  s^{(0)}(s^*,t)|
\stackrel{up}{\longrightarrow} 0$$ and $\inf_{t \in \mathcal{T}} s^{(0)}(s^*,t) > 0$; and
\item[(c)]
There exists a deterministic matrix function
$\mathbf{v}: \mathcal{I} \rightarrow \Re_+$ such that
$$\| \mathbb{Q}_n(s^*,w)\{[H^{(n)}(s^*,w)]^{\otimes 2}\} - \mathbf{v}(s^*,w) \|
\stackrel{up}{\longrightarrow} 0,$$
and for every $t \in (0,t^*]$,
\begin{displaymath}
\Sigma(s^*,t) = \int_0^t \mathbf{v}(s^*,w) s^{(0)}(s^*,w) \Lambda_0^{0}(dw)
\end{displaymath}
is positive definite.
\end{itemize}
Defining the stochastic integrals, for $n=1,2,\ldots$,
$$W^{(n)}(s^*,t) = \frac{1}{\sqrt{n}} \sum_{i=1}^n
\int_0^t H_i^{(n)}(s^*,w) M_i(s^*,dw),$$
then $\{W^{(n)}(s^*,t): t \in \mathcal{T}\}$ converges weakly on Skorohod's space
$D[0,t^*]$ to a zero-mean Gaussian process $\{W^{(\infty)}(s^*,t): t \in \mathcal{T}\}$
whose covariance function is
\begin{displaymath}
\mbox{Cov}\{W^{(\infty)}(s^*,t_1),W^{(\infty)}(s^*,t_2)\} =
\Sigma(s^*,\min(t_1,t_2)).
\end{displaymath}
\end{theo}

\begin{proof}
The proof of this result is analogous to the proof of the general theorem in
\cite{PenStrHol00}.
\end{proof}

\section{Consistency Properties}
\label{sec-consistency}

In this section we will establish the consistency of the sequence of estimators
$\hat{\eta}_n$ and $\hat{\Lambda}_n(s^*,\cdot)$ as the number of units $n$
increases to infinity.

We shall assume the following set of ``regularity conditions.''

\begin{enumerate}

\item[(C1)]
For each $(s,t) \in \mathcal{I}$,
$\eta \mapsto \kappa(s,t;\eta)$ is twice-continuously differentiable with
$$\stackrel{.}{\kappa}(s,t;\eta) = \nabla_\eta \kappa(s,t;\eta)
\quad \mbox{and} \quad
\stackrel{..}{\kappa}(s,t;\eta) = \nabla_{\eta\eta\trp} \kappa(s,t;\eta).$$
Furthermore, the operations of differentiation (with respect to $\eta$) and integration
could be interchanged.

\item[(C2)]
There exists a deterministic function $s^{(0)}: \mathcal{I} \times \Gamma \rightarrow
\Re_+$ such that
\begin{displaymath}
\sup_{t \in \mathcal{T}; \eta \in \Gamma}
|S^{(0)}(s^*,t;\eta) - s^{(0)}(s^*,t;\eta)|
\stackrel{p}{\longrightarrow} 0,
\end{displaymath}
and with $\inf_{t \in \mathcal{T}} s^{(0)}(s^*,t;\eta) > 0$ and with $\Lambda_0^0(t^*)
= \int_0^{t^*} \lambda_0^0(w) dw < \infty$.

\item[(C3)]
There exist deterministic functions $s^{(1)}: \mathcal{I} \times \Gamma^2 \rightarrow
\Re^k$ and $s^{(2)}: \mathcal{I} \times \Gamma^2 \rightarrow (\Re^k)^{\otimes 2}$ such
that with
\begin{eqnarray*}
Q_n^{(1)}(s^*,t;\eta_1,\eta_2) & = &
\mathbb{Q}_n(s^*,t;\eta_1)\left[\frac{\stackrel{.}{\kappa}}{\kappa}
(\mathcal{E}^{-1}(t);\eta_2)\right]; \\
Q_n^{(2)}(s^*,t;\eta_1,\eta_2) & = &
\mathbb{Q}_n(s^*,t;\eta_1)\left[\frac{\stackrel{..}{\kappa}}{\kappa}
(\mathcal{E}^{-1}(t);\eta_2)\right],
\end{eqnarray*}
and
\begin{eqnarray*}
q^{(1)}(s^*,t;\eta_1,\eta_2) & = & \frac{s^{(1)}}{s^{(0)}}(s^*,t;\eta_1,\eta_2); \\
q^{(2)}(s^*,t;\eta_1,\eta_2) & = & \frac{s^{(2)}}{s^{(0)}}(s^*,t;\eta_1,\eta_2),
\end{eqnarray*}
we have
\begin{eqnarray*}
& \sup_{t \in \mathcal{T}; (\eta_1,\eta_2) \in \Gamma^2}
\left\|
Q_n^{(1)}(s^*,t;\eta_1,\eta_2) - q^{(1)}(s^*,t;\eta_1,\eta_2)
\right\|  \stackrel{p}{\longrightarrow}  0; & \\
& \sup_{t \in \mathcal{T}; (\eta_1,\eta_2) \in \Gamma^2}
\left\|
Q_n^{(2)}(s^*,t;\eta_1,\eta_2) -
q^{(2)}(s^*,t;\eta_1,\eta_2)
\right\|  \stackrel{p}{\longrightarrow}  0. &
\end{eqnarray*}

\item[(C4)]
With
$\mathbf{v}(s^*,t)$ satisfying
\begin{displaymath}
\sup_{t \in \mathcal{T}} \left\| \mathbb{V}_{\mathbb{Q}_n(s^*,t)}
\left[\frac{\stackrel{.}{\kappa}}{\kappa}
\left(\mathcal{E}^{-1}(t)\right)\right] -
\mathbf{v}(s^*,t)\right\| \stackrel{pr}{\longrightarrow} 0,
\end{displaymath}
the matrix
\begin{displaymath}
\mathbf{\Sigma}(s^*,t)  =
\int_0^t \mathbf{v}(s^*,w) s^{(0)}(s^*,w) \Lambda_0^0(dw)
\end{displaymath}
is positive definite for each $t \in (0,t^*]$.
\item[(C5)]
For each $s \in [0,s^*]$, the mappings
\begin{eqnarray*}
(v,\eta) & \mapsto &
\frac{\stackrel{.}{\kappa}}{\kappa}(v;\eta) - Q_n^{(1)}(s,\mathcal{E}(v);\eta,\eta); \\
(v,\eta) & \mapsto &
\frac{\stackrel{..}{\kappa}}{\kappa}(v;\eta) - Q_n^{(2)}(s,\mathcal{E}(v);\eta,\eta),
\end{eqnarray*}
are bounded and $\mathfrak{F}_{s-}$-measurable for each $v \in [0,s]$.
\end{enumerate}

We first establish an intermediate result.

\begin{lemm}
\label{lemma-s0}
For $w \in \mathcal{T}$ and $\eta \in \Gamma$, we have
\begin{eqnarray*}
\frac{\stackrel{.}{S}^{(0)}}{S^{(0)}}(s^*,w;\eta) & = & Q_n^{(1)}(s^*,w;\eta,\eta); \\
\frac{\stackrel{..}{S}^{(0)}}{S^{(0)}}(s^*,w;\eta) & = & Q_n^{(2)}(s^*,w;\eta,\eta).
\end{eqnarray*}
\end{lemm}

\begin{proof}
The proofs are straightforward and hence omitted.
\end{proof}

For notational brevity, let us define
\begin{eqnarray*}
& \Psi_n(s^*,t^*;\eta)  =
\nabla_\eta \left\{\frac{1}{n} {l}_P(s^*,t^*;\eta)\right\}; & \\
& \Psi(s^*,t^*;\eta)  =  \int_0^{t^*}
\left[
q^{(1)}(s^*,w;\eta^0,\eta) - q^{(1)}(s^*,w;\eta,\eta)
\right]
s^{(0)}(s^*,w) \Lambda_0^0(dw), &
\end{eqnarray*}
where
%
${l}_P(s^*,t^*;\eta) = \log L_P(s^*,t^*;\eta)$
%
is the logarithm of the partial likelihood function.
We are now in position to state a result concerning the consistency of the
partial MLE of $\eta$. Without loss of generality,
we shall assume that the maximizer of the partial
likelihood can be obtained as a zero of $\eta \mapsto \Psi_n(s^*,t^*;\eta)$.

\begin{theo}
If $\hat{\eta}_n$ is such that $\Psi_n(s^*,t^*;\hat{\eta}_n) = 0$ and if,
for every $\epsilon > 0$, we have that
\begin{displaymath}
\inf_{\{\eta: ||\eta - \eta^0|| \ge \epsilon\}}
||\Psi(s^*,t^*;\eta)|| > 0,
\end{displaymath}
then, under the regularity conditions (C1)--(C5),
$\hat{\eta}_n \stackrel{p}{\longrightarrow} \eta^0$.
\end{theo}

\begin{proof}
From (\ref{profile likelihood}), (C1), and Lemma \ref{lemma-s0}, we have
\begin{eqnarray}
\Psi_n(s^*,t^*;\eta) & = & \mathbb{P}_n \int_0^{s^*}
\left[
\frac{\stackrel{.}{\kappa}}{\kappa}(v;\eta) - Q_n^{(1)}(s^*,\mathcal{E}(v);\eta,\eta)
\right]
N(dv,t^*) \nonumber \\
& = & \mathbb{P}_n \int_0^{s^*}
\left[
\frac{\stackrel{.}{\kappa}}{\kappa}(v;\eta) - Q_n^{(1)}(s^*,\mathcal{E}(v);\eta,\eta)
\right]
M(dv,t^*) + \label{mart-part} \\
&& \mathbb{P}_n \int_0^{s^*}
\left[
\frac{\stackrel{.}{\kappa}}{\kappa}(v;\eta) - Q_n^{(1)}(s^*,\mathcal{E}(v);\eta,\eta)
\right]
A(dv,t^*). \label{comp-part}
\end{eqnarray}
By (C5) and Theorem \ref{theo-master theorem}, the term in (\ref{mart-part}) is
$o_p(1)$. On the other hand, the term in (\ref{comp-part}) becomes, after splitting
the region of integration into the disjoint intervals
$(S_{j-1},S_j]$ for $j = 1, 2, \ldots, N^\dagger(s^*-)+1$ and then
doing a variable transformation,
\begin{eqnarray*}
\mbox{Term (\ref{comp-part})} & = & \int_0^{t^*}
\mathbb{P}_n \left\{
\sum_{j=1}^{N^\dagger(s^*-)+1}
\left[
\frac{\stackrel{.}{\kappa}}{\kappa}(\mathcal{E}_j^{-1}(w);\eta) -
Q_n^{(1)}(s^*,w;\eta,\eta)
\right] \times \right. \\
&& \left. \varphi_j[\mathcal{E}_j^{-1}(w);\eta]
I_{(\mathcal{E}(S_{j-1}),\mathcal{E}(S_j)]}(w)
\right\}
\Lambda_0^0(dw) \\
& = & \int_0^{t^*}
S^{(0)}(s^*,w)
\left[Q_n^{(1)}(s^*,w;\eta^0,\eta) - Q_n^{(1)}(s^*,w;\eta,\eta)\right]
\Lambda_0^0(dw).
\end{eqnarray*}
By conditions (C2) and (C3), this last term will converge uniformly in
probability to $\Psi(s^*,t^*;\eta)$, so that we will have the result
\begin{equation}
\label{limit of score}
\sup_{\eta \in \Gamma}
\left\|\Psi_n(s^*,t^*;\eta) - \Psi(s^*,t^*;\eta)\right\|
\stackrel{p}{\longrightarrow} 0.
\end{equation}
Finally, observe that $\Psi(s^*,t^*;\eta^0) = 0$, so by the condition of
the theorem and coupling with (\ref{limit of score}), it follows from Theorem
5.9 of van der Vaart \cite{Vaa98}) that $\hat{\eta}_n \stackrel{p}{\longrightarrow}
\eta^0$.
\end{proof}

Indeed, there is more to be said based on the following Lemma \ref{lemm-result about second derivative of likelihood}
which will also be used in the weak convergence result proof in Section \ref{sec-weak convergence}.
Since $\mathbf{\Sigma}(s^*,t^*)$
is positive definite, this lemma implies that, in fact, $\eta^0$ is a maximizer of the
limit in probability of the log-partial likelihood
$[l_P(s^*,t^*;\eta) - l_P(s^*,t^*)]/n$.

\begin{lemm}
\label{lemm-result about second derivative of likelihood}
Under conditions (C1)-(C5),
\begin{eqnarray*}
\lefteqn{\stackrel{\cdot}{\Psi}_n(s^*,t^*) \equiv
\nabla_{\eta\eta\trp}
\left\{
\frac{1}{n} l_P(s^*,t^*;\eta)
\right\}|_{\eta=\eta^0} } \\
& = & -\int_0^{t^*}
\mathbb{V}_{\mathbb{Q}_n(s^*,w)}\left[
\frac{\stackrel{.}{\kappa}}{\kappa}\left(\mathcal{E}^{-1}(w)\right)
\right] S^{(0)}(s^*,w) \Lambda_0^0(dw) + o_p(1) \\
& \stackrel{p}{\longrightarrow} & - \mathbf{\Sigma}(s^*,t^*).
\end{eqnarray*}
\end{lemm}

\begin{proof}
Straightforward, though tedious, calculations show that
\begin{eqnarray*}
\lefteqn{ \stackrel{\cdot}{\Psi}_n(s^*,t^*;\eta)  =
\mathbb{P}_n \int_0^{s^*}
\left[
\frac{\stackrel{\cdot\cdot}{\kappa}}{\kappa}(v;\eta) -
\frac{\stackrel{\cdot\cdot}{S}^{(0)}}{S^{(0)}}\left(s^*,\mathcal{E}(v);\eta\right)
\right] N(dv,t^*) - } \\
&& \mathbb{P}_n \int_0^{s^*}
\left\{
\left[\frac{\stackrel{\cdot}{\kappa}}{\kappa}(v;\eta)\right]^{\otimes 2} -
\left[\frac{\stackrel{\cdot}{S}^{(0)}}{S^{(0)}}\left(s^*,\mathcal{E}(v);\eta\right)\right]^{\otimes 2}
\right\} N(dv,t^*) \\
& = &
\int_0^{t^*}
\left\{
Q_n^{(2)}(s^*,w;\eta^0,\eta) - Q_n^{(2)}(s^*,w;\eta,\eta)
\right\} S^{(0)}(s^*,w;\eta^0) \Lambda_0^0(dw) - \\
&& \int_0^{t^*}
\left\{
\mathbb{Q}_n(s^*,w;\eta^0) \left[
\frac{\stackrel{\cdot}{\kappa}}{\kappa}\left(\mathcal{E}^{-1}(w);\eta\right)\right]^{\otimes 2} -
\left[ Q_n^{(1)}(s^*,w;\eta,\eta) \right]^{\otimes 2}
\right\} \times \\
&&
S^{(0)}(s^*,w;\eta^0) \Lambda_0^0(dw) + o_p(1).
\end{eqnarray*}
Evaluating at $\eta = \eta^0$, and noting that
\begin{displaymath}
\mathbb{Q}_n(s^*,w;\eta^0) \left[
\frac{\stackrel{\cdot}{\kappa}}{\kappa}\left(\mathcal{E}^{-1}(w);\eta^0\right)\right]
=
Q_n^{(1)}(s^*,w;\eta^0,\eta^0)
\end{displaymath}
then yields the representation given in the statement of the lemma.
Letting $n \rightarrow \infty$, the limiting matrix is $-\mathbf{\Sigma}(s^*,t^*)$.
\end{proof}

\begin{theo}
\label{theo-consistency of lambda-hat}
Under conditions (C1)-(C5), $\hat{\Lambda}_{0n}(s^*,\cdot)$ converges
uniformly in probability to $\Lambda_0^0(\cdot)$ on $[0,t^*]$, that is,
\begin{displaymath}
\sup_{t \in [0,t^*]} \left|\hat{\Lambda}_{0n}(s^*,t) - \Lambda_0^0(t)\right|
\stackrel{p}{\longrightarrow} 0.
\end{displaymath}
\end{theo}

\begin{proof}
With $\Lambda_0^*(s^*,t) = \int_0^t I\{S^{(0)}(s^*,w;\hat{\eta}) > 0\} \Lambda_0^0(dw)$,
we have that
\begin{eqnarray}
\lefteqn{
|\hat{\Lambda}_0(s^*,t) - \Lambda_0^0(t)| \le
|\hat{\Lambda}_0(s^*,t) - \Lambda_0^*(s^*,t)| +
|\Lambda_0^*(s^*,t) - \Lambda_0^0(t)| } \nonumber \\
& \le & \left|\hat{\Lambda}_0(s^*,t) - \int_0^t \frac{I\{S^{(0)}(s^*,w) > 0\}}{S^{(0)}(s^*,w)}
\mathbb{P} N(s^*,dw)\right| + \label{term1} \\
&& \left|\int_0^t \frac{I\{S^{(0)}(s^*,w) > 0\}}{S^{(0)}(s^*,w)} \mathbb{P} M(s^*,dw)\right| +
\label{term2} \\
&& \left|\int_0^t I\{S^{(0)}(s^*,w;\hat{\eta}) = 0\} \Lambda_0^0(dw)\right|.
\label{term3}
\end{eqnarray}
Term (\ref{term3}) is bounded above by
$$\left|\int_0^{t^*} I\{S^{(0)}(s^*,w;\hat{\eta}) = 0\} \Lambda_0^0(dw)\right|,$$
which is $o_p(1)$ since $S^{(0)}(s^*,w;\hat{\eta})
\stackrel{p}{\longrightarrow} s^{(0)}(s^*,w)$ and by (C2) we have
$\Lambda_0^0(t^*) < \infty$ and $\inf_{w \in [0,t^*]} s^{(0)}(s^*,w) > 0$.
Term (\ref{term1}) is bounded above by
\begin{displaymath}
\left\{\sup_{w \in [0,t^*]}\left|
\frac{I\{S^{(0)}(s^*,w;\hat{\eta}) > 0\}}{S^{(0)}(s^*,w;\hat{\eta})} -
\frac{I\{S^{(0)}(s^*,w) > 0\}}{S^{(0)}(s^*,w)}\right|\right\}
\mathbb{P} N(s^*,t^*).
\end{displaymath}
But $\mathbb{P} N(s^*,t^*) = \mathbb{P} M(s^*,t^*) + \mathbb{P} A(s^*,t^*)$.
By Theorem \ref{theo-master theorem}, $\mathbb{P} M(s^*,t^*) = o_p(1)$, while
$\mathbb{P} A(s^*,t^*) = \int_0^{t^*} S^{(0)}(s^*,w) \Lambda_0^0(dw)$, which
converges in probability to $\int_0^{t^*} s^{(0)}(s^*,w) \Lambda_0^0(dw)$,
a finite quantity by (C2). Thus, $\mathbb{P} N(s^*,t^*) = O_p(1)$. Since
\begin{displaymath}
\sup_{w \in [0,t^*]}\left|
\frac{I\{S^{(0)}(s^*,w;\hat{\eta}) > 0\}}{S^{(0)}(s^*,w;\hat{\eta})} -
\frac{I\{S^{(0)}(s^*,w) > 0\}}{S^{(0)}(s^*,w)}\right| = o_p(1)
\end{displaymath}
it therefore follows that term (\ref{term1}) is $o_p(1)$. Finally,
by Theorem \ref{theo-master theorem}, we have that the process
\begin{displaymath}
\left\{
\frac{1}{\sqrt{n}} \sum_{i=1}^n \int_0^t
\frac{I\{S^{(0)}(s^*,w) > 0\}}{S^{(0)}(s^*,w)} M_i(s^*,dw):\ t \in [0,t^*]
\right\}
\end{displaymath}
converges weakly to a zero-mean Gaussian process $G$ whose covariance function is
$$\mbox{Cov}(G(t_1),G(t_2)) = \int_0^{\min(t_1,t_2)} \frac{\Lambda_0^0(dw)}{s^{(0)}(s^*,w)}$$
for $t_1, t_2 \in [0,t^*]$. As a consequence,
\begin{displaymath}
\sup_{t \in [0,t^*]} \left| \frac{1}{\sqrt{n}} \sum_{i=1}^n \int_0^t
\frac{I\{S^{(0)}(s^*,w) > 0\}}{S^{(0)}(s^*,w)} M_i(s^*,dw) \right|
\end{displaymath}
converges weakly to $\sup_{t \in [0,t^*]} |G(t)|$, which is $O_p(1)$. It follows
that
\begin{eqnarray*}
\lefteqn{ \sup_{t \in [0,t^*]}
\left|\int_0^t \frac{I\{S^{(0)}(s^*,w) > 0\}}{S^{(0)}(s^*,w)} \mathbb{P} M(s^*,dw)\right| } \\
& = & \frac{1}{\sqrt{n}} \sup_{t \in [0,t^*]} \left|
\frac{1}{\sqrt{n}} \sum_{i=1}^n \int_0^t
\frac{I\{S^{(0)}(s^*,w) > 0\}}{S^{(0)}(s^*,w)} M_i(s^*,dw) \right| \\
& = & o_p(1).
\end{eqnarray*}
This completes the proof of the theorem.
\end{proof}

\section{Distributional Properties}
\label{sec-weak convergence}

In this section we establish the limiting distributional properties of
$\{\sqrt{n}[\hat{\eta}_n - \eta^0], n=1,2,\ldots\}$ and $\{W_n(s^*,t): t \in \mathcal{T};
n= 1,2,\ldots\}$, where
\begin{displaymath}
W_n(s^*,t) = \sqrt{n} \left[
\hat{\Lambda}_0^{(n)}(s^*,t) - \Lambda_0^0(t)\right].
\end{displaymath}
Define the process $\{B_n(s^*,t): t \in \mathcal{T}; n=1,2,\ldots\}$ according to
\begin{displaymath}
B_n(s^*,t) = \int_0^{t}
I\{S^{(0)}(s^*,w) > 0\}
\frac{\stackrel{.}{S}^{(0)}(s^*,w)}{[S^{(0)}(s^*,w)]^2}
\mathbb{P}_n N(s^*,dw).
\end{displaymath}
Let us also define the process $\{V_n(s^*,t): t \in \mathcal{T}; n=1,2,\ldots\}$
via
\begin{displaymath}
V_n(s^*,t) = \sqrt{n}
\left[\hat{\Lambda}_0^{(n)}(s^*,t) - \Lambda_0^0(t)\right] +
\sqrt{n} (\hat{\eta}_n - \eta^0)\trp B_n(s^*,t).
\end{displaymath}
Furthermore, we shall assume that $\hat{\eta}_n$ solves the equation
\begin{displaymath}
U_P^{(n)}(s^*,t^*;\eta) = 0
\quad \mbox{with} \quad
U_P^{(n)}(s^*,t;\eta) = \nabla_\eta l_P(s^*,t;\eta).
\end{displaymath}
We now present and prove a result from which the asymptotic properties follow.

\begin{theo}
\label{theo-asymptotic respresentations}
Under conditions (C1)-(C5), we have the representations
\begin{eqnarray}
\lefteqn{ \sqrt{n}(\hat{\eta}_n - \eta^0) =
\left[
\Sigma(s^*,t^*)
\right]^{-1} \times } \nonumber \\ &&
\left\{
\sqrt{n} \mathbb{P}_n
\int_0^{t^*}
\left[
\frac{\stackrel{.}{\kappa}}{\kappa}
\left[
\mathcal{E}^{-1}(w)
\right] -
\frac{\stackrel{.}{S}^{(0)}}{S^{(0)}}(s^*,w)
\right]
M(s^*,dw)
\right\}
+ o_p(1);
\label{asy rep of eta hat}
\end{eqnarray}
and
\begin{equation}
\label{asy rep Lambda hat}
V_n(s^*,t) = \sqrt{n} \int_0^{t^*}
I(w \le t)
\frac{I\{S^{(0)}(s^*,w) > 0\}}{S^{(0)}(s^*,w)}
\mathbb{P}_n M(s^*,dw) + o_p(1).
\end{equation}
Furthermore, $\{\sqrt{n}(\hat{\eta}_n - \eta^0)\}$ and
$\{V_n(s^*,t): t \in \mathcal{T}\}$ are asymptotically independent
with each weakly converging to Gaussian limits.
\end{theo}

\begin{proof}
From the definition of $\hat{\eta}_n$, we have by first-order Taylor expansion that
\begin{displaymath}
\sqrt{n}(\hat{\eta}_n - \eta^0) =
\left[
-\stackrel{\cdot}{\Psi}_n(s^*,t^*;\tilde{\eta}_n)
\right]^{-1}
\left[
\sqrt{n} \Psi_n(s^*,t^*;\eta^0)
\right]
\end{displaymath}
where $\tilde{\eta}_n$ is in a neighborhood centered at $\eta^0$ and whose radius is $||\hat{\eta}_n - \eta^0||$.
It is easy to see that
\begin{eqnarray*}
\sqrt{n} \Psi_n(s^*,t^*;\eta^0) & = &
\sqrt{n} \mathbb{P}_n \int_0^{s^*}
\left\{
\frac{\stackrel{\cdot}{\kappa}}{\kappa}(v) -
\frac{\stackrel{\cdot}{S}^{(0)}}{S^{(0)}}(s^*,\mathcal{E}(v))
\right\} M(dv,t^*) \\
& = &
\sqrt{n} \mathbb{P}_n \int_0^{t^*}
\left\{
\frac{\stackrel{\cdot}{\kappa}}{\kappa}[\mathcal{E}^{-1}(w)] -
\frac{\stackrel{\cdot}{S}^{(0)}}{S^{(0)}}(s^*,w)
\right\} M(s^*,dw).
\end{eqnarray*}
Furthermore, since $\hat{\eta}_n \stackrel{p}{\rightarrow} \eta^0$, and by virtue of
Lemma \ref{lemm-result about second derivative of likelihood}, we have that
\begin{displaymath}
\left[
-\stackrel{\cdot}{\Psi}_n(s^*,t^*;\tilde{\eta}_n)
\right]^{-1}
=
[\mathbf{\Sigma}(s^*,t^*)]^{-1} + o_p(1).
\end{displaymath}
As such we obtain the representation for $\sqrt{n}(\hat{\eta}_n - \eta^0)$.

Once again, by first-order Taylor expansion, we have that on the set where
$S^{(0)}(s^*,w;\hat{\eta}_n) > 0$,
\begin{displaymath}
\frac{1}{S^{(0)}(s^*,w;\hat{\eta}_n)} =
\frac{1}{S^{(0)}(s^*,w;{\eta}^0)} -
(\hat{\eta}_n - \eta^0)\trp \frac{\stackrel{\cdot}{S}^{(0)}(s^*,w;\tilde{\eta}_n)}
{[S^{(0)}(s^*,w;\tilde{\eta}_n)]^2}
\end{displaymath}
with $\tilde{\eta}_n$ inside the ball centered at $\eta^0$ with radius $||\hat{\eta}_n - \eta^0||$.
Defining
\begin{displaymath}
\Lambda_0^*(s^*,t) = \int_0^t I\{S^{(0)}(s^*,w;\hat{\eta}_n) > 0\} \Lambda_0^0(dw),
\end{displaymath}
and recalling that
\begin{displaymath}
\hat{\Lambda}_0^{(n)}(s^*,t) = \int_0^t \frac{I\{S^{(0)}(s^*,w;\hat{\eta}_n) > 0\}}
{S^{(0)}(s^*,w;\hat{\eta}_n)} \mathbb{P}_n N(s^*,dw),
\end{displaymath}
we obtain
\begin{eqnarray*}
\lefteqn{ \sqrt{n} \left[
\hat{\Lambda}_0^{(n)}(s^*,t) - \Lambda_0^*(s^*,t) \right] =
\int_0^t \frac{I\{S^{(0)}(s^*,w;\hat{\eta}_n) > 0\}}
{S^{(0)}(s^*,w;\hat{\eta}_n)} \sqrt{n} \mathbb{P}_n M(s^*,dw) - } \\
&  & \sqrt{n}(\hat{\eta}_n - \eta^0)\trp
\int_0^t I\{S^{(0)}(s^*,w;\hat{\eta}_n) > 0\}
\frac{[\stackrel{\cdot}{S}^{(0)}(s^*,w;\tilde{\eta}_n)]}
{[S^{(0)}(s^*,w;\tilde{\eta}_n)]^2}
\mathbb{P}_n N(s^*,dw).
\end{eqnarray*}
The representation for $V_n(s^*,t)$ given in the statement of the lemma now follows by noting that
\begin{eqnarray*}
& \sup_{0 \le t \le t^*} \|\sqrt{n}[\Lambda_0^*(s^*,t) - \Lambda_0^0(t)\| = o_p(1); & \\
& \sup_{0 \le t \le t^*} \|S^{(0)}(s^*,t;\hat{\eta}_n) -  S^{(0)}(s^*,t;\eta^0)\| = o_p(1); \\
& \sup_{0 \le t \le t^*} \|\stackrel{\cdot}{S}^{(0)}(s^*,t;\hat{\eta}_n) - \stackrel{\cdot}{S}^{(0)}(s^*,t;\eta^0)\| = o_p(1). &
\end{eqnarray*}
Finally, let $\mathbf{t} = (t_1,t_2,\ldots,t_p)\trp \subset \mathcal{T}$. From the just-established representations,
with $I\{w \le \mathbf{t}\} = (I\{w \le t_1\}, \ldots, I\{w \le t_p\})\trp$, we have
\begin{eqnarray*}
\lefteqn{\left[
\begin{array}{c}
\sqrt{n}(\hat{\eta}_n - \eta^0) \\
V_n(s^*,\mathbf{t})
\end{array}
\right]
 =
\left[
\begin{array}{cc}
\mathbf{\Sigma}(s^*,t^*)^{-1} & \mathbf{0} \\
\mathbf{0} & \mathbf{I}
\end{array}
\right] \times } \\
&&
\sqrt{n}\mathbb{P}_n \int_0^{t^*}
\left[
\begin{array}{c}
\frac{\stackrel{\cdot}{\kappa}}{\kappa}[\mathcal{E}^{-1}(w)] -
\frac{\stackrel{\cdot}{S}^{(0)}}{S^{(0)}}(s^*,w) \\
I(w \le \mathbf{t}) \frac{I\{S^{(0)}(s^*,w) > 0\}}{S^{(0)}(s^*,w)}
\end{array}
\right]
M(s^*,dw) + o_p(1).
\end{eqnarray*}
By the main weak convergence theorem or by invoking the Martingale Central Limit
Theorem after a time transformation, this converges weakly to the random vector
\begin{displaymath}
\left[
\begin{array}{cc}
\mathbf{W}_1 \\
\mathbf{W}_2
\end{array}
\right]
=
\left[
\begin{array}{cc}
\mathbf{\Sigma}(s^*,t^*)^{-1} & \mathbf{0} \\
\mathbf{0} & \mathbf{I}
\end{array}
\right]
\left[
\begin{array}{c}
\mathbf{Z}_1 \\
\mathbf{Z}_2
\end{array}
\right]
\end{displaymath}
where $(\mathbf{Z}_1\trp, \mathbf{Z}_2\trp)\trp$ is a $(k+p)$-dimensional
zero mean multivariate normal random vector with covariance matrix
\begin{eqnarray*}
\lefteqn{\mbox{Cov}\left[
\left(\begin{array}{c}
\mathbf{Z}_1 \\ \mathbf{Z}_2
\end{array}\right),
\left(\begin{array}{c}
\mathbf{Z}_1 \\ \mathbf{Z}_2
\end{array}\right)
\right]   =
\mbox{plim}_{n\rightarrow\infty} \int_0^{t^*}
\mathbb{Q}_n(s^*,w) \times } \\ &&
\left[
\begin{array}{c}
\frac{\stackrel{\cdot}{\kappa}}{\kappa}[\mathcal{E}^{-1}(w)] -
\frac{\stackrel{\cdot}{S}^{(0)}}{S^{(0)}}(s^*,w) \\
I(w \le \mathbf{t}) \frac{I\{S^{(0)}(s^*,w) > 0\}}{S^{(0)}(s^*,w)}
\end{array}
\right]^{\otimes 2}
S^{(0)}(s^*,w) \Lambda_0^0(dw).
\end{eqnarray*}
However, the covariance matrix between $\mathbf{Z}_1$ and $\mathbf{Z}_2$
equals $\mathbf{0}$ since, for every $w \in \mathcal{T}$,
\begin{displaymath}
\mathbb{Q}_n(s^*,w)
\left[
\frac{\stackrel{\cdot}{\kappa}}{\kappa}[\mathcal{E}^{-1}(w)] -
\frac{\stackrel{\cdot}{S}^{(0)}}{S^{(0)}}(s^*,w)
\right] = \mathbf{0}.
\end{displaymath}
Because of the Gaussian limits, this then establishes that $\sqrt{n}(\hat{\eta} - \eta^0)$ and $V_n(s^*,\cdot)$ are
asymptotically independent.
\end{proof}

The following two corollaries are then immediate consequences of the preceding
theorem and elements of its proof.

\begin{coro}
\label{coro-asymptotics of eta-hat}
Under the conditions of Theorem \ref{theo-asymptotic respresentations},
as $n \rightarrow \infty$,
\begin{displaymath}
\sqrt{n}(\hat{\eta}_n - \eta^0) \stackrel{d}{\longrightarrow}
N\left(0, \Sigma(s^*,t^*)^{-1}\right).
\end{displaymath}
\end{coro}

\begin{proof}
This is immediate from the fact that $\mathbf{Z}_1$ in the proof of
Theorem \ref{theo-asymptotic respresentations} is a $k$-dimensional
zero-mean normal vector with covariance matrix $\mathbf{\Sigma}(s^*,t^*)$.
\end{proof}

\begin{coro}
\label{coro-asymptotics of Lambdahat}
Under the conditions of Theorem \ref{theo-asymptotic respresentations},
as $n \rightarrow \infty$, the process $W_n(s^*,\cdot) = \sqrt{n}
\left[ \hat{\Lambda}_0^{(n)}(s^*,\cdot) - \Lambda_0^0(\cdot) \right]$
converges weakly in Skorohod's $D[\mathcal{T}]$-space to a zero-mean
Gaussian process with covariance function given by
\begin{equation}
\label{covariance of limit process}
c(s^*,t_1,t_2) = \int_0^{\min(t_1,t_2)}
\frac{\Lambda_0^0(dw)}{s^{(0)}(s^*,w)} +
b(s^*,t_1)\trp \{\Sigma(s^*,t^*)\}^{-1} b(s^*,t_2),
\end{equation}
for $t_1, t_2 \in \mathcal{T}$ and with
$b(s^*,t) = \int_0^t q^{(1)}(s^*,w) \Lambda_0^0(dw)$.
\end{coro}

\begin{proof}
From Theorem \ref{theo-asymptotic respresentations} we have the results that
$$\sqrt{n}(\hat{\eta}_n - \eta^0) \stackrel{d}{\rightarrow} \mathbf{W}_1(s^*,t^*)$$
where $\mathbf{W}_1(s^*,t^*) \sim N(\mathbf{0},[\mathbf{\Sigma}(s^*,t^*)]^{-1})$.
Also, we have that
$$\{V_n(s^*,t): t \in \mathcal{T}\} \Rightarrow \{{Z}_2(s^*,t): t \in \mathcal{T}\}$$
where $\{{Z}_2(s^*,t): t \in \mathcal{T}\}$ is a zero-mean Gaussian process with
covariance function
$$Cov\{Z_2(s^*,t_1),Z_2(s^*,t_2)\} = \int_0^{\min(t_1,t_2)} \frac{\Lambda_0^0(dw)}{s^{(0)}(s^*,w)}.$$
In addition, $\mathbf{W}_1(s^*,t^*)$ and $\{Z_2(s^*,t): t \in \mathcal{T}\}$ are
independent. It is also evident that
$$\sup_{t \in \mathcal{T}} ||B_n(s^*,t) - b(s^*,t)|| \stackrel{p}{\rightarrow} 0.$$
From the representations in Theorem \ref{theo-asymptotic respresentations}, it follows that
$\{W_n(s^*,t): t \in \mathcal{T}\}$ converges weakly to the process
$W_\infty \equiv \{W_\infty(s^*,t): t \in \mathcal{T}\}$
with
$$W_\infty(s^*,t) = Z_2(s^*,t) - b(s^*,t)\trp \mathbf{W}_1(s^*,t^*).$$
As such $W_\infty$ is a zero-mean Gaussian process and its covariance function is
\begin{eqnarray*}
c(s^*,t_1,t_2) & = & Cov\{W_\infty(s^*,t_1),W_\infty(s^*,t_2)\} \\
& = &
\int_0^{\min(t_1,t_2)} \frac{\Lambda_0^0(dw)}{s^{(0)}(s^*,w)} +
b(s^*,t_1)\trp [\mathbf{\Sigma}(s^*,t^*)]^{-1} b(s^*,t_2).
\end{eqnarray*}
This completes the proof of the corollary.
\end{proof}

Possible consistent estimators of the covariance functions are then easily
obtained. For the covariance matrix $\mathbf{\Sigma}(s^*,t^*)$, this could be
estimated by
\begin{eqnarray*}
\hat{\mathbf{\Sigma}}(s^*,t^*) & = &
\int_0^{t^*}
\mathbb{Q}_n(s^*,w;\hat{\eta}_n)
\left[
\frac{\stackrel{\cdot}{\kappa}}{\kappa}[\mathcal{E}^{-1}(w);\hat{\eta}_n] -
\frac{\stackrel{\cdot}{S}^{(0)}}{S^{(0)}}(s^*,w;\hat{\eta}_n)
\right]^{\otimes 2} \times \\ &&
S^{(0)}(s^*,w;\hat{\eta}_n) \hat{\Lambda}_0^{(n)}(s^*,dw;\hat{\eta}_n).
\end{eqnarray*}
For the covariance function of $Z_2(s^*,\cdot)$, a consistent estimator is
given by
\begin{displaymath}
\widehat{Cov}[Z_2(s^*,t_1),Z_2(s^*,t_2)] = \int_0^{\min(t_1,t_2)}
\frac{\hat{\Lambda}_0^{(n)}(s^*,dw)}{S^{(0)}(s^*,w;\hat{\eta}_n)}.
\end{displaymath}
On the otherhand, an estimator of $b(s^*,t)$ is given by
\begin{displaymath}
\hat{b}(s^*,t) = \int_0^t \frac{\stackrel{\cdot}{S}^{(0)}(s^*,w;\hat{\eta}_n)}
{S^{(0)}(s^*,w;\hat{\eta}_n)} \hat{\Lambda}_0^{(n)}(s^*,dw).
\end{displaymath}
From these estimators, we are then able to obtain a consistent estimator of the
covariance function $c(s^*,t_1,t_2)$ of the limiting Gaussian process $W_\infty(s^*,\cdot)$.
This estimator is
\begin{displaymath}
\hat{c}(s^*,t_1,t_2) = \widehat{Cov}[Z_2(s^*,t_1),Z_2(s^*,t_2)] +
\hat{b}(s^*,t_1)\trp [\hat{\mathbf{\Sigma}}(s^*,t^*)]^{-1} \hat{b}(s^*,t_1).
\end{displaymath}

Observe that the results in Corollaries \ref{coro-asymptotics of eta-hat}
and \ref{coro-asymptotics of Lambdahat} are highly analogous to those in
\cite{AndGil82} pertaining to the estimators of the parameters of the
Cox proportional hazards model. However, one need to be cautious since under
the setting being considered, the limit functions appearing in the above results
are more complicated as they must reflect aspects of the sum-quota
accrual scheme and the dynamics of the performed interventions or repairs after
each event occurrence.

Through these asymptotic results, large-sample confidence intervals and bands,
large-sample hypothesis testing procedures, and goodness-of-fit or model validation methods for the
infinite-dimensional parameters may now be constructed for this general dynamic model for
recurrent events.
We note, however, that the results presented in this paper are still limited to the general dynamic recurrent
event model {\em without} frailties. It remains an open problem to obtain large-sample results
for the general dynamic model incorporating frailties.







\bibliography{JNSSubmission}
\bibliographystyle{plain}

\end{document}